\documentclass[a4paper,reqno]{amsart}

\usepackage[utf8]{inputenc}
\usepackage[english]{babel}
\usepackage[T1]{fontenc}
\usepackage{microtype}
\usepackage{amsmath,amssymb,amsthm}
\usepackage{thmtools,thm-restate}
\usepackage{mathtools}
\usepackage{upgreek}
\usepackage{orcidlink}
\usepackage{hyperref}
\hypersetup{hidelinks}
\hypersetup{breaklinks=true}
\usepackage[nameinlink]{cleveref}
\usepackage{suffix}

\usepackage{enumitem} 
\usepackage{amsaddr} 

\title[{A simple way to well-posedness of a DDE from cell biology}]{A simple way to well-posedness in $H^{1}$ of a delay differential equation from cell biology}
\author[B.~Aigner]{Bernhard Aigner\orcidlink{0009-0009-8252-162X}}
\thanks{supported by the state of Saxony via a graduate student stipend}
\address[B.~Aigner]{Institut f\"{u}r Angewandte Analysis\\
  TU Bergakademie Freiberg\\
  Institut für Angewandte Analysis\\
  Pr\"{u}ferstr. 9, 09599 Freiberg\\
  Germany}
\email[B.~Aigner]{bernhard.aigner@doktorand.tu-freiberg.de}

\author[M.~Waurick]{Marcus Waurick\orcidlink{0000-0003-4498-3574}}
\address[M.~Waurick]{Institut f\"{u}r Angewandte Analysis\\
  TU Bergakademie Freiberg\\
  Institut für Angewandte Analysis\\
  Pr\"{u}ferstr. 9, 09599 Freiberg\\
  Germany}
\email[M.~Waurick]{marcus.waurick@math.tu-freiberg.de}
\date{\today}

\DeclarePairedDelimiterX{\norm}[1]{\lVert}{\rVert}{#1}
\DeclarePairedDelimiterX{\abs}[1]{\lvert}{\rvert}{#1}
\DeclarePairedDelimiterX{\scprod}[2]{(}{)}{#1\delimsize| #2}
\newcommand{\e}{\mathrm{e}} 
\newcommand{\dd}{\mathrm{d}} 
\newcommand{\dx}[1][x]{\,\dd{}#1}

\newtheorem*{definition*}{Definition}
\newtheorem*{theorem*}{Theorem}
\newtheorem*{proposition*}{Proposition}
\newtheorem*{lemma*}{Lemma}
\newtheorem*{corollary*}{Corollary}
\newtheorem*{remark*}{Remark}

\newtheorem{definition}{Definition}[section]
\newtheorem{theorem}[definition]{Theorem}
\newtheorem{proposition}[definition]{Proposition}
\newtheorem{lemma}[definition]{Lemma}
\newtheorem{corollary}[definition]{Corollary}
\newtheorem{remark}[definition]{Remark}
\newtheorem{assumptions}[definition]{Assumptions}

\begin{document}

\begin{abstract}
  We present an application of recent well-posedness results in the theory of delay differential equations for ordinary differential equations~\cite{Waurick2023} to a generalized population model for stem cell maturation. The weak approach using Sobolev-spaces we take allows for a larger class of initial prehistories and makes checking the requirements for well-posedness of such a model considerably easier compared to previous approaches. In fact the present approach is a possible means to guarantee that the solution manifold is not empty, which is a necessary requirement for a $\mathcal{C}^{1}$-approach to work.
\end{abstract}

\maketitle

\section{Introduction}
A \textit{delay differential equation (DDE)} is a differential equation (or system of differential equations) in which the derivative not only depends on the current state of the system, but on a prior state or even on a prior time horizon of the state variable. Specifically we are interested in differential equations of the form
\begin{alignat*}{3}
  x'(t) &= G(t,x_{t}) &\quad& \text{for}\ t \geq 0, \\
  x(t) &= \Phi(t) &&\text{for}\ t \in [-h,0]
\end{alignat*}
where we use the standard notation $x_{t}(s)=x(t+s)$ for $s<0$. $G$ is at this point any right-hand side that depends potentially on a prior state of $x$ or even a longer history $x_{t}$. Therefore the system above is referred to as a \textit{state-dependent} DDE. Even more specifically this system can be regarded as the equivalent of an ordinary differential equation for DDEs, since the left-hand side is just the time derivative of the state. In the setting of (state-dependent) DDEs a so-called prehistory  $\Phi$ takes the place of a traditional initial value in IVPs to prevent systems from being unterdetermined. For information on general results and techniques for DDEs we refer to \cite{Diekmann1995, Corduneanu2016}.\\
The present article is an application of a recent finding from \cite{Waurick2023} to a generalized population model for stem cell growth that has been introduced in \cite{Doumic2011} and has subsequently been studied in articles such as \cite{Getto2016, Getto2021}. We will prove well-posedness of the mathematical model as a system of DDEs. In doing so we improve on previous results. Indeed in \cite{Getto2016} substantially more regularity on the individual functions is required, whereas in \cite{Getto2021}, complementing \cite{Getto2016}, the structural hypothesis of positivity needs to be imposed on initial prehistories. Here we simultaneously generalize both of these results. For a detailed comparison see section two for \cite{Getto2016} and section four for \cite{Getto2021}. We emphasize that we do not need to include the criteria necessary to apply the concept of solution manifolds as introduced in \cite{Walther2003}, which is useful in the context of linearized stability, but introduces complications in the context of solution theory. Thus \cite{Walther2003} may rather be viewed as a regularity theory for DDEs.\\
We start by (re-)introducing the model in the next section as well as stating all assumptions on the model. In section three we present our results with full proofs. Section four contains a brief discussion and comparison to other research.

\section{The Model}
We start by immediately presenting the model we want to investigate, then explain its contents and elaborate on where it comes from. The model of concern is the following system of differential equations
\begin{equation}
  \label{fullPDE}
  \begin{split}
    w'(t)&=q(v(t))w(t) \\
    v'(t)&=\tfrac{\gamma(v(t-\tau(v_{t}))) g(x_{2},v(t))w(t-\tau(v_{t}))}{g(x_{1},v(t-\tau(v_{t})))} \e^{\int_{0}^{\tau(v_{t})}(d-D_{1}g)(y(s,v_{t}),v(t-s))\dx[s]}-\mu v(t)
  \end{split}
\end{equation}
where we understand $v_{t}\colon \,[-h,0]\rightarrow \mathbb{R}, \,s \mapsto v(t+s)$. Notice the terms $t-\tau(v_{t})$ in the second equation. The delay $\tau$ depending on a prior history of the state $v_{t}$ maps to values in $[0,h]$, where $h$ is the backward time horizon. The times $t-\tau(v_{t})$ therefore denote instances in the past. This system appears in \cite{Getto2016}, is based on the model used in \cite{Doumic2011} and is obtained from the original transport equation there by integration along the characteristics. For details we refer to both articles. The system describes the growth of stem cells, originally in mammary glands. We provide a full table of all functions and parameters involved before we give a brief heuristic of the biology involved.
\begin{center}
  \begin{tabular}{|l|l|}
    \hline
    \textbf{Symbol} & \textbf{Description} \\
    \hline
    \hline
    $w\colon I \rightarrow \mathbb{R}$ & concentration of stem cells \\
    \hline
    $v\colon I \rightarrow \mathbb{R}$ & concentration of mature cells \\
    \hline
    $q\colon I \rightarrow \mathbb{R}^{+}$ & stem cell population net growth rate \\
    \hline
    $\gamma\colon I \rightarrow \mathbb{R}^{+}_{0}$ & unregulated maturation rate of stem cells \\
    \hline
    $g \colon J\times I \rightarrow \mathbb{R}^{+}_{0}$ & regulated maturation rate of progenitor cells \\
    \hline
    $y\colon [-h,0]\times \mathcal{M} \rightarrow \mathbb{R}$ & population size of mature cells \\
    \hline
    $d \colon J\times I\rightarrow \mathbb{R}^{+}_{0}$ & per cell net growth rate/progenitor net production rate \\
    \hline
    $\tau\colon \mathcal{M} \rightarrow \mathbb{R}^{+}_{0}$ & delay function \\
    \hline
    $x_{1}\in \mathbb{R}$ & initial maturity \\
    \hline
    $x_{2}\in \mathbb{R}, x_{1}<x_{2}$ & full maturity \\
    \hline
    $x\in [x_{1},x_{2}]$ & maturity of stem cells \\
    \hline
    $\mu \geq 0$ & decay rate/mortality rate of mature cells \\
    \hline
  \end{tabular}
\end{center}
Here $I\subseteq \mathbb{R}$ is an open interval, $J$ is an open interval containing $[x_{1},x_{2}]$ and $\mathcal{M}$ is a suitable function space of functions mapping $[-h,0]$ into $I$ ($H^{1}(-h,0;I)$ - see \cpageref{H1} for a definition - in this publication and $\mathcal{C}^{1}([-h,0],I)$ in other sources).\\
For a given function $\varphi$ defined on an interval $[-h,0]$ the function $y(.,\varphi)=y_{\varphi}$ is a solution of the ordinary differential equation
\begin{equation}
  \label{ODEy}
  \begin{split}
    y_{\varphi}'(s)&=-g(y_{\varphi}(s),\varphi(-s)) \\
    y_{\varphi}(0)&=x_{2}
  \end{split}
\end{equation}
Note that $y_{\varphi}$ is defined going backward in time. $s$ then can be interpreted as the time it takes to elvolve $y_{\varphi}(s)$ to $x_{2}$. The delay term $\tau(\varphi)$ is subsequently defined as the solution of
\begin{equation}
  \label{Eqtau}
  y_{\varphi}(\tau(\varphi))=x_{1}.
\end{equation}
$\tau$ and $y_{\varphi}$ are well-defined as we will clarify after stating our assumptions. Prior to that we would like to devote some space briefly outlining what the model describes.

\subsection{The biological setting}
The first line in \Cref{fullPDE} is almost self-explanatory: the current change in the concentration of stem cells is assumed to be proportional to the current concentration of stem cells. The corresponding proportionality rate is the stem cell population growth rate, which in turn is assumed to exclusively depend on the concentration of mature cells. $q$ reflects stem cell division rate in the absence of mature cells, self-renewal in the absence of mature cells, mortality rate of stem cells and regulation constants. For details we refer to \cite{Alarcon2011}, where the authors represent $q$ in a more explicit form. In particular, $q$ as modeled there is a bounded $\mathcal{C}^{\infty}$-function.\\
The second line in \Cref{fullPDE} asks for  more explanation: First consider the progenitor phase, which is the phase in which stem cells commit themselves to maturation. Let us consider the amount of progenitor cells passing through this level per time. The inflow of progenitors into the mature cells is assumed to be linear with respect to the concentration of stem cells times a factor depending on the concentration of mature cells. At time $t$ this flow is $\beta(v(t))w(t)$. We again refer to \cite{Alarcon2011} for a more detailed explanation and expression of $\beta$. We note that $\beta$ as introduced there is a bounded $\mathcal{C}^{\infty}$-function again. The inflow of progenitors into the mature cell population is modelled with a delay. Asssuming a stem cell just enters the progenitor phase and does neither divide nor die, it fully matures in a finite time $\tau$. The maturation process is regulated by mature cells, hence it is assumed that the delay only depends on the history of the concentration of mature cells $v_{t}$. Thus if full maturity of a cell is reached in time $t$, $\tau(v_{t})$ is the time spent as progenitor cells and $t-\tau(v_{t})$ is the time the progenitor phase was entered. This in turn means that the inflow of stem cells into the progenitor phase that corresponds to those cells that fully mature at time $t$ is
\begin{equation*}
  \beta(v(t-\tau(v_{t})))w(t-\tau(v_{t})),
\end{equation*}
that is, if it is assumed that the number of cells would not change during the progenitor phase. This change in numbers is accounted for by a progenitor net population growth factor $\mathcal{F}(\varphi)$, where $\varphi$ denotes the history experienced by the time cells reach full maturity. Now one only needs to account for the outflow of mature cells, which is the number of mature cells that die at a given time, which is assumed to be a constant mortality rate $\mu$ mulitiplied by the current concentration of mature cells. All in all one arrives at the change of concentration of mature cells
\begin{equation*}
  v'(t)=\beta(v(t-\tau(v_{t})))w(t-\tau(v_{t}))\mathcal{F}(v_{t})-\mu v(t).
\end{equation*}
The specific expression for $\mathcal{F}$ can be obtained by introducing the population net growth rate $d$ yielding the equation
\begin{alignat*}{2}
  &\tilde{\mathcal{F}}' (s) = d(\text{some maturity level},\text{some prehistory})\tilde{\mathcal{F}} (s) &\quad& t>0 \\
  &\tilde{\mathcal{F}} (0) = 1 && {}
\end{alignat*}
and setting $\mathcal{F}(\psi)\coloneq \tilde{\mathcal{F}}(\tau(\psi))$. We will not go into details here and are content with the observation that this equation gives rise to an exponential term.\\
It remains to justify \Cref{ODEy}. This ODE is immediately obtained by introducing the function $g$ as the regulated maturation speed of a progenitor cell, which is simply the change of the number of mature cells. For more details, explanations and explicit expressions we refer to \cite{Alarcon2011}.\\
Models of this sort apparently arise generally in the modelling of stem cell growth models (to the best of our knowledge this model first arose in the context of mammary stem cell growth), but appear to have wider applicability, for instance in the modelling of tumor growth. For more context we refer to \cite{Doumic2011} and the references therein. We note that the functions provided there have explicit expressions, are smooth and bounded. We will study a more general mathematical scenario in the following.

\subsection{The mathematical model}
We will not investigate the model as in \Cref{fullPDE}, but rather an equivalent formulation that is more suitable for our approach, that has been studied for instance in \cite{Getto2016,Getto2021} and is very close to the outline of the deduction of the model anyway. We also wish to account for more mathematical generality. For this purpose we denote
\begin{alignat}{3}
  ~\beta &\colon \, I \rightarrow \mathbb{R}^{+} &\quad& z\mapsto \frac{\gamma(z)}{g(x_{1},z)} \notag\\
  \mathcal{G} &\colon \, \mathcal{M} \rightarrow \mathbb{R}^{+} && \psi\mapsto g(x_{2},\psi(0))\e^{\int_0^{\tau(\psi)} (d-D_{1}g))(y(t,\psi),\psi(-s)) \dx[s]} \label{DefG}
  \end{alignat}
Using these functions we can define
\begin{equation*}
  \begin{split}
    F:\,H^{1}(-h,0;I^{2})&\rightarrow \mathbb{R}^{2} \\
    (\varphi,\psi) &\mapsto \begin{pmatrix}
      q(\psi(0))\varphi(0) \\
      \beta(\psi(-\tau(\psi)))\varphi(-\tau(\psi))\mathcal{G}(\psi)-\mu\psi(0) \end{pmatrix}
  \end{split}
\end{equation*}
We understand the \textit{Sobolev space} $H^{1}$ with values in $U\subseteq \mathbb{R}^{n}$ as
\begin{equation*}
  \label{H1}
  H^{1}(-h,0;U)\coloneq \{f\in L^{2}(-h,0;U)\colon \, u'\in L^{2}(-h,0;U)\}
\end{equation*}
where the weak (time) derivative is defined in the variational sense and the space is equipped with the norm
\begin{equation*}
  \norm{u}_{H^{1}(-h,0;U)}\coloneq \sqrt{\norm{u}_{L^{2}(-h,0;U)}^{2}+\norm{u'}_{L^{2}(-h,0;U)}^{2}}
\end{equation*}
We assume familiarity with the concept of weak derivatives and Sobolev spaces, for the uninitiated good first resources are \cite[ch.~5.2]{Evans2010} and \cite[ch.~6.1]{Waurick2022}.\\
Utilizing this new function we can rewrite \Cref{fullPDE} as $x'(t)=F(x_{t})$, substituting $x=(w,v)$. Henceforth we will only investigate the following initial value problem (IVP)
\begin{equation}
  \label{IVP}
  \begin{matrix}
    \begin{split}
      x'(t) &= F(x_{t}) \\
      x_{0}(t) &= \Phi(t) \\
    \end{split}
    &
    \begin{split}
      t&\geq 0 \\
      t&\in [-h,0]\\
    \end{split}
  \end{matrix}
\end{equation}
It should be noted, that using Sobolev spaces a priori begs the question how one would interpret the equations above, in particular the initial condition. This is in fact no problem at all though, since the equalities can be interpreted pointwise, because at least in the one dimensional case, every (equivalence class of a) function in $H^{1}$ admits a unique continuous representative:
\begin{theorem}[Sobolev-embedding]
  \label{Sobolev}
  Let $-\infty<a<b<\infty$. Then every $f\in H^{1}(a,b)$ admits a continuous representative and
  \begin{equation*}
    \norm{f}_{\mathcal{C}([a,b])}\leq ((b-a)^{\frac{1}{2}}+(b-a)^{-\frac{1}{2}})\norm{f}_{H^{1}(a,b)}
  \end{equation*}
\end{theorem}
This is a standard result found in any standard reference, e.g. \cite[thm.~4.12]{Adams2003}. For a proof of we refer to \cite[thm.~4.9]{Arendt2015}. We shall now state our assumptions that we will assume throughout the remainder of this article.
\begin{assumptions}[Basic assumptions]
  \label{Assumptions}
  Let $I$ and $J$ be open intervals with $J$ containing $[x_{1},x_{2}]$. We assume
  \begin{enumerate}[leftmargin=3ex]
  \item $q, \beta\colon \,I\rightarrow \mathbb{R}$ are locally Lipschitz-continuous.
  \item $d\colon \,J\times I\rightarrow \mathbb{R}$ is continuous and locally Lipschitz-continuous with respect to the second component.
  \item $g\colon \,J\times I\rightarrow \mathbb{R}$ for some $b,\epsilon,K\in \mathbb{R}$ satisfies:
  \begin{itemize}[leftmargin=3ex]
    \item $g$ is Lipschitz-continuous and partially differentiable with respect to the first component.
    \item $D_{1}g$ is locally Lipschitz-continuous with respect to the second component.
    \item $0<\epsilon \leq g(x,y)\leq K$ on $\overline{\mathcal{B}(x_{2},b)} \times I$ and $x_{2}-x_{1}\in (0,\frac{b}{K}\epsilon)$.
  \end{itemize}
  \end{enumerate}
  Let $h\coloneqq\frac{b}{K}$.
\end{assumptions}
A remark on notation: In this article the notation $\mathcal{B}_{X}(x,r)$ will always refer to the open ball around some element $x\in X$ of a Banach space $X$ with radius $r>0$. When we want to refer to the (optimal) Lipschitz-constant $L$ of a function $f\colon X\rightarrow \mathbb{R}^{d}$ on a particular subset $C\subseteq X$ we write $\norm{f}_{\operatorname{Lip},C}\coloneq \underset{x,y\in C}{\operatorname{sup}}\frac{\norm{f(x)-f(y)}}{\abs{x-y}}$. When the spaces are clear we will drop the indices.\\
To prove global existence later we will impose the additional requirements that we will list here for the sake of compactness.
\begin{assumptions}[Further assumptions]
  \label{Furhter assumptions}
  In addition to \Cref{Assumptions} let
  \begin{itemize}[leftmargin=3ex]
    \item $q$ be bounded.
    \item $d-D_{1}g$ be bounded.
    \item $\beta$ satisfy a linear growth bound.
  \end{itemize}
\end{assumptions}
The boundedness of $g$ as in the \Cref{Assumptions} is always assumed in this model, cf. \cite{Getto2016,Doumic2011}. It certainly seems to make sense from a biological point of view to assume bounds on the maturation rate $g$; after all it seems reasonable to assume that a certain minimum number of cells always die (and need therefore be replaced). At the same time assuming an upper bound is obvious. We point out that the backward time horizon $h$ is not arbitrary. It needs to be at least $\frac{b}{K}$ for \Cref{ODEy} to have a solution, which we will argue in due time.\\
We point out that this model has been studied before and local existence and uniqueness of solutions has been shown, albeit under different and sometimes considerably stronger assumptions. To showcase the substantial weakening in our assumption and for comparison's sake we list the assumptions in \cite{Getto2016}, which we will frequently treat as a comparison article throughout our investigation. A comparison to the more up to date article \cite{Getto2021} we reserve for the last section.
\begin{assumptions}[Assumptions in {\cite[thm.~1.13]{Getto2016}}]
  \label{OldAssumptions}
  Let $I$ be an open interval containing $0$ and $J$ be an open interval containing $[x_{1},x_{2}]$. Assume
  \begin{enumerate}[leftmargin=3ex]
  \item $d, D_{1}g$ are continuously differentiable.
  \item $\beta, q$ are continuously differentiable.
  \item $g$ satisfies for some $K,\epsilon,b\in \mathbb{R}$:
  \begin{itemize}[leftmargin=3ex]
    \item $\overline{\mathcal{B}(x_{2},b)} \subseteq J$ and $g:\, J\times I \rightarrow \mathbb{R}$ is continuously differentiable.
    \item $D_{1}g(x,y)$ is bounded on $\overline{\mathcal{B}(x_{2},b)} \times I$
    \item $0<\epsilon \leq g(x,y)\leq K$ on $\overline{\mathcal{B}(x_{2},b)} \times I$ and $x_{2}-x_{1}\in (0,\frac{b}{K}\epsilon)$
  \end{itemize}
  \item \upshape{(solution manifold condition):}\\Let $\Phi \in \mathcal{C}^{1}([-h,0],\mathbb{R}_{+}^{2})$ satisfy $F(\Phi)=\Phi'(0)$.
  \end{enumerate}
  For global existence suppose further:
  \begin{enumerate}[leftmargin=3ex]
  \setcounter{enumi}{4}
    \item $\operatorname{sup}_{(x,z)\in\overline{\mathcal{B}(x_{2},b)}\times I}\abs{D_{1}g(x,z)}<\frac{K}{b}$
    \item $d$ is bounded on $\overline{\mathcal{B}(x_{2},b)}\times I$.
    \item $D_{2}g, D_{1}d, D_{2}d, D_{1}D_{1}g, D_{2}D_{1}g$ are bounded on $\overline{\mathcal{B}(x_{2},b)}\times A$ whenever $A\subseteq I$ is bounded.
    \item $\gamma, q$ are Lipschitz-continuous on bounded sets and bounded.
  \end{enumerate}
\end{assumptions}
There are two major differences comparing these two sets of assumptions:  For one our regularity assumptions are weaker, and more importantly, we do not require the solution manifold assumption. To provide some context: The concept of the so-called \textit{solution manifold} (cf. \cite{Walther2003}) was introduced for the study of linearized stability in DDEs. Without going into details the rough idea in this context is to consider for some open $U\subseteq \mathcal{C}^{1}([-h,0],\mathbb{R}^{n})$ the set of consistent initial histories $X\coloneq \{\varphi\in U\colon \varphi'(0)=F(\varphi)\}$ for $F$ at least continuous. Then a necessary condition for the solvability of \Cref{IVP} is that $X$ is non-empty. Under certain differentiability assumptions it can be shown that then $X$ is a continuously differentiable submanifold of $U$ and for each $\varphi\in X$ there exists $t_{\varphi}>0$ and a unique non-continuable solution $x_{\varphi}\colon [-h,t_{\varphi})\rightarrow \mathbb{R}^{n}$ of the IVP \Cref{IVP}. We refer to \cite[thm~3.2.1]{Hartung2006}. The condition of $X$ being non-empty must be fulfilled in \Cref{OldAssumptions} to obtain even local existence of solutions in \cite{Getto2016}. This is why in said article the authors demand $0\in I$ as this condition serves to show that the zero function is contained in the solution manifold. Here we are able to disperse this condition altogether; this improves applicability of our well-posedness result. Since we are able to also remove condition (5), we provide a more accessible approach for numerical computations as it is virtually impossible to calculate suprema with a computer.\\

Using our assumptions we can immediately observe that at least the differential equation \Cref{ODEy} is well-posed and that both $y$ and $\tau$ are well-defined.
\begin{remark}[$\tau$ and $y$ are well-defined]
  \label{ywelldefined}
  We observe that \Cref{ODEy} for a given $\varphi \in \mathcal{C}([-h,0];I)$ satisfies the condition of the Picard-Lindel\"{o}f theorem (cf. \cite[thm.~2.2]{Teschl2012}) since $g$ is Lipschitz-continuous with respect to the first variable (interpret the RHS $-g(y(s),\varphi(s))$ as $\tilde{g}(y(s),s)\coloneq -g(y(s),\varphi(s))$) independent of the second variable. The solution is therefore also unique. $\tau$ then is implicitly defined via \Cref{Eqtau}. To see that such $\tau$ is well-defined note
  \begin{itemize}[leftmargin=3ex]
    \item $y(0)=x_{2}$ and the derivative $y'$ is strictly decreasing with a slope between $-\epsilon$ and $-K$ because of the assumption on $g$.
    \item Since by assumption $x_{2}-x_{1}\in(0,\frac{b}{K}\epsilon)$ the solution $y$ of \Cref{ODEy} satisfies a Lipschitz-condition on $\overline{\mathcal{B}(x_{2},b)}$, the solution exists at least up to $x_{1}$ appealing to the global version of the Picard-Lindel\"{o}f theorem (cf. \cite[thm.~2.2]{Teschl2012}).
  \end{itemize}
  Then appealing to the intermediate value theorem for continuous functions yields existence of $\tau(\varphi)$ and, because $y$ is strictly decreasing, also uniqueness.
\end{remark}
We also observe that $y$ is quite well-behaved in general:
\begin{remark}[properties of $y$]
  \label{yexpbdd}
  Note that $y$ as a solution of the ODE \Cref{ODEy} is continuously differentiable for a fixed history $\varphi\in H^{1}(-h,0;\mathbb{R})$ and we obtain an exponential growth bound for $y$:
  \begin{align*}
    \abs{y(t)}&\leq \abs{y(0)}+\abs[\bigg]{\int_0^t y'(s)\dx[s]}\\
              &= x_{2}+\abs[\bigg]{\int_0^t -g(y(s),\varphi(-s))}\dx[s]\\
              &\leq x_{2}+ \abs{g(x_{2},\varphi(0)}t + \int_0^t \abs[\big]{g(y(s),\varphi(-s))-g(x_{2},\varphi(0)}\dx[s]\\
              &\leq x_{2}+ \abs{g(x_{2},\varphi(0)}t + \int_0^t L\left(\abs{y(s)-x_{2}}+\abs{\varphi(-s)-\varphi(0)}\right)\dx[s]\\
              &\leq x_{2}+ \left(\abs{g(x_{2},\varphi(0)} + 2\norm{\varphi}_{\mathcal{C}([-h,0],\mathbb{R})}+L x_{2}\right)t + \int_0^t L\abs{y(s)}\dx[s]
  \end{align*}
  Appealing to Gr\"{o}nwall's lemma we obtain the estimate
  \begin{equation*}
    \abs{y(t)}\leq \left[x_{2}+ \left(\abs{g(x_{2},\varphi(0)} + 2\norm{\varphi}_{\mathcal{C}([-h,0],\mathbb{R})}+L x_{2}\right)t\right]\e^{Lt}
  \end{equation*}
  This means in particular that we have global existence of $y$ at least as long the histories remain finite in sup-norm. Utilizing the same technique we can also obtain a (local) Lipschitz-constant for $y$ depending on the histories:
  \begin{align*}
    \abs{y_{\varphi}(t)-y_{\psi}(t)}&= \abs[\bigg]{\int_0^t y_{\varphi}'(s)-y_{\psi}'(s)\dx[s]}\\
              &= \abs[\bigg]{\int_0^t g(y_{\psi}(s),\psi(-s))-g(y_{\varphi}(s),\varphi(-s))}\dx[s]\\
              &\leq \int_0^t L\left(\abs{y_{\varphi}(s)-y_{\psi}(s)}+\abs{\varphi(-s)-\psi(-s)}\right)\dx[s]\\
              &\leq L \norm{\varphi-\psi}_{\mathcal{C}([-h,0],\mathbb{R})}t + \int_0^t L\abs{y_{\varphi}(s)-y_{\psi}(s)}\dx[s]
  \end{align*}
  and therefore
  \begin{equation*}
    \abs{y_{\varphi}(t)-y_{\psi}(t)}\leq L \norm{\varphi-\psi}_{\mathcal{C}([-h,0],\mathbb{R})}t\e^{Lt}
  \end{equation*}
\end{remark}

\section{Results}
Our approach follows the orthodox route in the study of differential equations. We first would like to solve \Cref{IVP} in small neighbourhoods of the initial value (prehistory in our case), which is to say prove local existence and uniqueness of a solution there. After settling that matter, we will then proceed to maximal solutions at the end of this section. The key for local existence is the use of \cite[Thm~5.1]{Waurick2023}, that we will restate below. It can be viewed as a generalized version of the Picard-Lindel\"{o}f theorem for Hilbert spaces using an adapted version of Lipschitz-continuity, that we introduce first:
\begin{definition}
  A function $G\colon [0,T]\times H^{1}(-h,0;\mathbb{R}^{n})\rightarrow \mathbb{R}^{n}$ is called \textup{almost uniformly Lipschitz-continuous} if it is continuous and for all $\alpha>0$, there exists $L\geq0$ such that for all $u,v\in V_{\alpha}\coloneqq \{\psi \in H^{1}(-h,0;\mathbb{R}^{n}):\, \norm{\psi'}_{\infty}\leq \alpha \}$ and $t\in [0,T]$:
  \begin{equation*}
    \abs{G(t,u)-G(t,v)}_{\mathbb{R}^{n}}\leq L\norm{u-v}_{H^{1}(-h,0;\mathbb{R}^{n})}\
  \end{equation*}
\end{definition}
In other words almost uniform Lipschitz-continuity means that we demand regular/global Lipschitz-continuity with respect to the second variable, but only for those variables with bounded derivative (uniformly in time of course). Now we can restate from \cite{Waurick2023}:
\begin{theorem}
  \label{localFDE}
  Let $n\in \mathbb{N}$ and let $G:[0,T]\times H^{1}(-h,0;\mathbb{R}^{n})\rightarrow\mathbb{R}^{n}$ be almost uniformly Lipschitz-continuous. Then for all $\Phi\in H^{1}(-h,0;\mathbb{R}^{n})$ with bounded derivative there exists $0 < T_{0}\leq T$ and a unique $x:[-h,T_{0})\rightarrow\mathbb{R}^{n}$ such that for all $T_{1}\in(0,T_{0})$, the restriction $x|_{(-h,T_{1})}\in H^{1}(-h,T_{1};\mathbb{R}^{n})$ satisfies:
  \begin{alignat}{1}
    \begin{aligned}
    \label{FDE}
    x'(t)&=G(t,x_{t}) &\quad& t\in(0,T_{1}) \\
    x_{0}&=\Phi.
    \end{aligned}
  \end{alignat}
  Such an $x$ we call an \textup{$H^{1}$-solution} of \Cref{FDE}.
\end{theorem}
The proof of this theorem is simply an adaption of Morgenstern's proof for the Picard-Lindel\"{o}f theorem, utilizing exponentially weighted Sobolev spaces, see \cite{Waurick2023} for a full proof. \\
\Cref{localFDE} now gives us a clear goal: What we need to prove for local existence and uniqueness of \Cref{IVP} is that $F$ as defined in the previous section is almost uniformly Lipschitz-continuous as a map from $H^{1}([-h,0],\mathbb{R}^{2})$ to $\mathbb{R}^{2}$. We state the main result:
\begin{theorem}[Local existence theorem]
  \label{localExistence}
  For a given $\Phi \in H^{1}(-h,0;\mathbb{R}^{2})$ with $\norm{\Phi'}_{\infty}<\infty$ there exists $T>0$ and a unique $x\colon [-h,T)\rightarrow \mathbb{R}^{2}$ such that for all $0<T_{0}<T$ the restriction $x|_{[-h,T_{0}]}\in H^{1}(-h,T_{0};\mathbb{R}^{2})$ and $x$ is a solution to the IVP \Cref{IVP} in the sense of \Cref{localFDE}.
\end{theorem}
It has already been shown that $F$ admits a Lipschitz-continuity property (denoted by (sLb)) on bounded subsets of $\mathcal{C}^{1}([-h,0];\mathbb{R}^{2})$ in \cite[Lem~2.2]{Getto2016}, albeit under the stronger \Cref{OldAssumptions} utilizing different techniques.
In order to apply \Cref{localFDE} above we need to restrict ourselves not just to $V_{\alpha}$ but a $\norm{.}_{\infty}$-ball around the prehistory $\Phi$ as well as we shall see in the proof of our main theorem. This is not a problem though as we can always extend a function defined on such a small set to the whole space as the following result shows:
\begin{proposition}[{\cite[thm.~6.1]{Waurick2023}}]
  \label{Restriction}
  Let $f\colon \mathcal{C}([-h,0];\mathbb{R}^{n}) \supseteq \operatorname{dom}(f)\rightarrow \mathbb{R}^{m}$ be almost uniformly Lipschitz-continuous with $\operatorname{dom}(f) \subseteq \mathcal{C}([-h,0];\mathbb{R}^{n})$ open. Then for all Lipschitz-continuous $\phi_{0}\in \operatorname{dom}(f)$ there exists $F\colon H^{1}(-h,0;\mathbb{R}^{n})\rightarrow \mathbb{R}^{m}$ Lipschitz continuous and $\delta>0$ such that for all $\alpha>0$ there holds
  \begin{equation*}
    F|_{\mathcal{B}_{\mathcal{C}([-h,0])}(\phi_{0},\delta)\cap V_{\alpha}} = f|_{\mathcal{B}_{\mathcal{C}([-h,0])}(\phi_{0},\delta)\cap V_{\alpha}}
  \end{equation*}
\end{proposition}

Before we prove our main result we have to get some preliminaries out of the way and in particular we would like to remind the reader of some well-known facts, such as the following:
\begin{remark}[Fundamental theorem of calculus]
  \label{FTC}
  Since our $H^{1}$-functions are differentiable almost everywhere (cf. \cite[sec.~5.8.3]{Evans2010}) and because the Sobolev embedding \Cref{Sobolev} says that $H^{1}$-functions on a finite interval are continuous, they adhere to the fundamental theorem of calculus (cf. \cite[thm.~6.3.10]{Cohn1980}).
\end{remark}
We shall employ this fact in the following two lemmata:
\begin{lemma}
  \label{evLip}
  For any $\alpha>0$ the restricted evaluation mappping
  \begin{align*}
    \operatorname{ev}\colon \left(H^{1}(-h,0;\mathbb{R})\cap V_{\alpha}\right)\times [-h,0]&\rightarrow \mathbb{R} \\
    (\psi,s)&\mapsto\psi(s)
  \end{align*}
   is Lipschitz-continuous.
\end{lemma}
\begin{proof}
  For $\varphi,\psi\in H^{1}(-h,0;\mathbb{R})$ and $s,t\in[-h,0]$ we compute utilizing \Cref{FTC}:
  \begin{align*}
    \abs{\operatorname{ev}(\varphi,s)-\operatorname{ev}(\psi,t)}&=\abs{\varphi(s)-\psi(t)} \\
                                                                &\leq \abs{\varphi(s)-\varphi(t)}+\abs{\varphi(t)-\psi(t)} \\
                                                                &\leq \abs[\bigg]{\int_s^t \varphi'(\sigma) \dx[\sigma]} + \norm{\varphi - \psi}_{\mathcal{C}([-h,0],\mathbb{R})} \\
                                                                &\leq \abs{t-s}^{\frac{1}{2}}\left(\int_s^t \abs{\varphi'(\sigma)}^{2}\dx[\sigma]\right)^{\frac{1}{2}} + (\abs{h}^{\frac{1}{2}}+\abs{h}^{-\frac{1}{2}}) \norm{\varphi - \psi}_{H^{1}(-h,0;\mathbb{R})} \\
                                                                &\leq \abs{t-s}^{\frac{1}{2}}\alpha\abs{t-s}^{\frac{1}{2}} + (h^{\frac{1}{2}}+h^{-\frac{1}{2}}) \norm{\varphi - \psi}_{H^{1}(-h,0;\mathbb{R})}\\
                                                                &\leq \operatorname{max}\{\alpha, h^{\frac{1}{2}}+h^{-\frac{1}{2}}\}(\abs{t-s}+\norm{\varphi - \psi}_{H^{1}(-h,0;\mathbb{R})})
  \end{align*}
\end{proof}
\begin{lemma}
  \label{ThetaLip}
  For a given function $u\in H^{1}(-h,T;\mathbb{R}^{n})$ the function
  \begin{align*}
    \Theta \colon [0,T]&\rightarrow H^{1}(-h,0;\mathbb{R}^{n})\\
    t&\mapsto u_{t}
  \end{align*}
  is Lipschitz continuous.
\end{lemma}
\begin{proof}
  Let $0\leq t < s\leq T$. Appealing to \Cref{FTC} we can estimate:
  \begin{align*}
    \norm{\Theta(t)-\Theta(s)}_{H^{1}(-h,0;\mathbb{R}^{n})}^{2} &= \int_{-h}^0 \abs{u_{t}(\sigma)-u_{s}(\sigma)}^{2}\dx[\sigma]\\
                                                                &= \int_{-h}^0 \abs{u(t+\sigma)-u(s+\sigma)}^{2}\dx[\sigma]\\
                                                                &= \int_{-h}^0 \abs[\bigg]{\int_{s+\sigma}^{t+\sigma}u'(\rho) \dx[\rho]}^{2}\dx[\sigma]\\
                                                                &\leq \int_{-h}^0 \abs{s-t}\int_{s+\sigma}^{t+\sigma}\abs{u'(\rho)}^{2} \dx[\rho] \dx[\sigma]\\
                                                                &=\int_{s-h}^{t} \abs{s-t} \abs{u'(\rho)}^{2}\int_{\rho-t}^{\rho-s} \dx[\sigma] \dx[\rho]\\
                                                                &=\abs{s-t}^{2} \int_{s-h}^{t} \abs{u'(\rho)}^{2} \dx[\rho]\\
                                                                &\leq \abs{s-t}^{2}\norm{u}_{H^{1}(-h,T;\mathbb{R}^{n})}^{2}
  \end{align*}
\end{proof}

For the proof of the main result we need to estimate the individual functions that make up $F$ from our IVP \Cref{IVP}. For convenience's sake we prove the necessary regularity properties in separate lemmata.
\begin{lemma}
  \label{tauLip}
  The function $\tau\colon \, H^{1}(-h,0;\mathbb{R})\rightarrow [0,h]$ is Lipschitz continuous.
\end{lemma}
\begin{proof}
  The proof is essentially the same as the proof of \cite[thm~6.2]{Waurick2023}. For the convenience of the reader we provide the details. Let $\varphi,\psi\in H^{1}(-h,0;\mathbb{R})$ and let $y_{\varphi},y_{\psi}$ denote the according to \Cref{ywelldefined} unique solutions determining the values $\tau(\varphi)$ and $\tau(\psi)$ respectively. Note that in that context the Sobolev embedding \Cref{Sobolev} is applicable. By integrating \Cref{ODEy} we obtain
  \begin{equation*}
    x_{2}-\int_{0}^{\tau(\varphi)} g(y_{\varphi}(t),\varphi(-t)) \dx[t] = x_{1} = x_{2}-\int_{0}^{\tau(\psi)} g(y_{\psi}(t),\psi(-t)) \dx[t]
  \end{equation*}
  Without loss of generality we assume $\tau(\psi)<\tau(\varphi)$ and using the \Cref{Assumptions} for $g$ we estimate
  \begin{align*}
    \epsilon\abs{\tau(\varphi)-\tau(\psi)} &\leq \int_{\tau(\psi)}^{\tau(\varphi)} g(y_{\psi}(t),\psi(-t)) \dx[t] \\
                                             &= \int_{0}^{\tau(\varphi)} g(y_{\psi}(t),\psi(-t)) \dx[t] - \int_{0}^{\tau(\psi)} g(y_{\psi}(t),\psi(-t)) \dx[t] \\
                                           &= \int_{0}^{\tau(\varphi)} g(y_{\psi}(t),\psi(-t)) \dx[t] - \int_{0}^{\tau(\varphi)} g(y_{\varphi}(t),\varphi(-t)) \dx[t] \\
                                           &\leq \int_{0}^{h} \abs[\big]{g(y_{\varphi}(t),\varphi(-t))-g(y_{\psi}(t),\psi(-t))}\dx[t] \\
                                           &\leq \int_{0}^{h} L\left(\abs{y_{\varphi}(t)-y_{\psi}(t)}+\abs{\varphi(-t)-\psi(-t)}\right)\dx[t]
  \end{align*}
  where $L$ is the Lipschitz-constant of $g$. Then, substituting the estimate
  \begin{equation*}
    \abs{y_{\varphi}(t)-y_{\psi}(t)}\leq L\norm{\varphi-\psi}_{\mathcal{C}([-h,0],\mathbb{R})}t\e^{Lt}
  \end{equation*}
  from \Cref{yexpbdd} into the inequality we conclude
  \begin{align*}
    \epsilon\abs{\tau(\varphi)-\tau(\psi)} &\leq \int_{0}^{h} \left(L^{2}\norm{\varphi-\psi}_{\mathcal{C}([-h,0],\mathbb{R})}t\e^{Lt} + L\abs{\varphi(-t)-\psi(-t)}\right)\dx[t] \\
                                           &= L^{2}\norm{\varphi-\psi}_{\mathcal{C}([-h,0],\mathbb{R})} \int_0^h te^{Lt} \dx[t] + \int_0^h L\abs{\varphi(-t)-\psi(-t)}\dx[t] \\
                                           &\leq L^{2} \int_{0}^{h} t\e^{Lt}\dx[t] \norm{\varphi-\psi}_{\mathcal{C}([-h,0];\mathbb{R})} + L \sqrt{h} \norm{\varphi-\psi}_{L^{2}(-h,0;\mathbb{R})} \\
                                             &\leq \left(L^{2}(h^{\frac{1}{2}}+h^{-\frac{1}{2}})\int_0^h t\e^{Lt} \dx[t] + L\sqrt{h}\right) \norm{\varphi-\psi}_{H^{1}(-h,0;\mathbb{R})}
  \end{align*}
  appealing to the Sobolev-embedding \Cref{Sobolev} again. This yields the claim.
\end{proof}

An immediate corollary from \Cref{evLip} and \Cref{tauLip} is the following:
\begin{corollary}
  \label{evtauLip}
  For every $\alpha>0$ the function
  \begin{equation*}
    \begin{split}
      \operatorname{ev}\circ (\operatorname{id}\times(-\tau))\colon \, H^{1}(-h,0;\mathbb{R})\cap V_{\alpha}&\rightarrow \mathbb{R}\\
      \varphi&\mapsto\varphi(-\tau(\varphi))
    \end{split}
  \end{equation*}
  is Lipschitz-continuous.
\end{corollary}

We will also deal with the integral term from $F$ separately.
\begin{lemma}
  \label{GLip}
  Let $\delta>0$ and $\phi \in H^{1}(-h,0;\mathbb{R})$. Then the function $\mathcal{G}$ from \Cref{DefG} as a function
  \begin{equation*}
    \mathcal{G}\colon \, \left( H^{1}(-h,0;\mathbb{R})\cap \mathcal{B}_{\mathcal{C}(-h,0;\mathbb{R})}(\phi,\delta) \right) \rightarrow \mathbb{R}
  \end{equation*}
  is Lipschitz-continuous.
\end{lemma}
\begin{proof}
  We remind ourselves of the form of this function:
  \begin{equation*}
    \mathcal{G}(\varphi) = g(x_{2},\varphi(0))\e^{\int_0^{\tau(\varphi)}(d-D_{1}g)\left(y(s,\varphi),\varphi(-s)\right)\dx[s]}
  \end{equation*}
  We will refer to the integral term in the exponent as $G(\varphi)$. We will estimate $G$ first. Now let $\varphi, \psi \in H^{1}(-h,0;\mathbb{R})\cap \mathcal{B}_{\mathcal{C}(-h,0;\mathbb{R})}(\phi,\delta)$. Without loss of generality we assume $\tau(\varphi)\leq \tau(\psi)$:
  \begin{align*}
    \MoveEqLeft
      \abs{G(\varphi)-G(\psi)} \\
      &\leq \abs[\bigg]{\int_0^{\tau(\varphi)}(d-D_{1}g)\left(y(s,\varphi),\varphi(-s)\right)\dx[s] - \int_0^{\tau(\psi)}(d-D_{1}g)\left(y(s,\psi),\psi(-s)\right)\dx[s]} \\
      &\leq \abs[\bigg]{\int_0^{\tau(\varphi)}(d-D_{1}g)\left(y(s,\varphi),\varphi(-s)\right) - (d-D_{1}g)\left(y(s,\psi),\psi(-s)\right)\dx[s]} \\
      &\quad  + \abs[\bigg]{\int_{\tau(\varphi)}^{\tau(\psi)}(d-D_{1}g)\left(y(s,\psi),\psi(-s)\right)\dx[s]} \\
      &\leq \int_0^{\tau(\varphi)}\abs[\big]{(d-D_{1}g)\left(y(s,\varphi),\varphi(-s)\right) - (d-D_{1}g)\left(y(s,\psi),\psi(-s)\right)}\dx[s] \\
      &\quad + \int_{\tau(\varphi)}^{\tau(\psi)}\abs[\big]{(d-D_{1}g)\left(y(s,\psi),\psi(-s)\right)}\dx[s]
  \end{align*}
  We claim that the function $d-D_{1}g$ is Lipschitz-continuous on
  \begin{equation*}
    N\coloneqq \overline{y([0,h]\times \mathcal{B}_{\mathcal{C}([-h,0],\mathbb{R})}(\phi,\delta))} \times [\phi(0)-M,\phi(0)+M],
  \end{equation*}
  where $M\coloneq \operatorname{sup}\{\norm{\zeta}_{\infty}\colon \, \zeta \in \mathcal{B}_{\mathcal{C}([-h,0],\mathbb{R})}(\phi,\delta)\}$. This holds because $y$ is continuous and exponentially bounded in this setting thanks to the growth estimate in \Cref{yexpbdd}. The conclusion follows from the fact that both $d$ and $D_{1}g$ are assumed to be locally Lipschitz-continuous and $N$ is bounded and closed, hence compact. We call the associated Lipschitz-constant $L_{k}$. The supremum of $d-D_{1}g$ on $N$ we denote as $M_{k}$. The Lipschitz-constant of $y$ with respect to the second component we denote as $L_{y}$ (cf. \Cref{yexpbdd}), the Lipschitz-constant of $\tau$ as $L_{\tau}$ (cf. \Cref{tauLip}). We continue estimating:
  \begin{equation*}
    \begin{split}
      &\abs{G(\varphi)-G(\psi)} \\
      &\leq \int_0^{\tau(\varphi)}\abs[\big]{(d-D_{1}g)(y(s,\varphi),\varphi(-s)) - (d-D_{1}g)(y(s,\psi),\psi(-s))}\dx[s] \\
      &\quad+ \int_{\tau(\varphi)}^{\tau(\psi)}\abs[\big]{(d-D_{1}g)(y(s,\psi),\psi(-s))}\dx[s] \\
      &\leq \int_0^{\tau(\varphi)}L_{k}\abs[\big]{(y(s,\varphi),\varphi(-s)) - (y(s,\psi),\psi(-s))}\dx[s] + M_{k} \abs{\tau(\varphi)-\tau(\psi)} \\
      &= \int_0^{\tau(\varphi)}L_{k}\sqrt{\abs{y(s,\varphi)-y(s,\psi)}^{2}+\abs{\varphi(-s)-\psi(-s)}^{2}}\dx[s] +M_{k} \abs{\tau(\varphi)-\tau(\psi)} \\
      &\leq \int_0^{\tau(\varphi)}L_{k}\sqrt{L_{y}^{2}\norm{\varphi-\psi}_{H^{1}(-h,0;\mathbb{R})}^{2}+\norm{\varphi-\psi}_{\mathcal{C}([-h,0],\mathbb{R})}^{2}}\dx[s] +M_{k}L_{\tau} \norm{\varphi-\psi}_{H^{1}(-h,0;\mathbb{R})} \\
      &\leq hL_{k}\sqrt{L_{y}^{2}+(\abs{h}^{\frac{1}{2}}+\abs{h}^{-\frac{1}{2}})^{2}}\norm{\varphi-\psi}_{H^{1}(-h,0;\mathbb{R})}+M_{k}L_{\tau} \norm{\varphi-\psi}_{H^{1}(-h,0;\mathbb{R})} \\
    \end{split}
  \end{equation*}
  where in the last step we made use of the Sobolev-embedding \Cref{Sobolev}. We can also bound $G$:
  \begin{equation*}
    \abs{G(\varphi)}=\abs[\bigg]{\int_0^{\tau(\varphi)}(d-D_{1}g)\left(y(s,\varphi),\varphi(-s)\right)\dx[s]} \leq \int_0^{\tau(\varphi)} M_{k} \leq hM_{k}
  \end{equation*}
  Now we can start to estimate $\mathcal{G}$:
  \begin{equation*}
    \begin{split}
      \abs{\mathcal{G}(\varphi)-\mathcal{G}(\psi)} &= \abs{g(x_{2},\varphi(0))\e^{G(\varphi)} - g(x_{2},\psi(0))\e^{G(\psi)}}\\
      &\leq \abs{g(x_{2},\varphi(0))}\abs{\e^{G(\varphi)} - \e^{G(\psi)}} + \abs{g(x_{2},\varphi(0)) - g(x_{2},\psi(0))}\abs{e^{G(\psi)}} \\
    \end{split}
  \end{equation*}
  Since $g$ is a continuous function and bounded on bounded sets we can estimate $\abs{g(x_{2},\xi(0))}$ with $M_{g}$. Since the exponential function is (locally) Lipschitz-continuous and $G$ is as well as shown above the composition is too with a Lipschitz-constant we denote by $\tilde{L}_{G}$. Hence the first term can be estimated as:
  \begin{equation*}
    \abs{g(x_{2},\varphi(0))}\abs{\e^{G(\varphi)} - \e^{G(\psi)}} \leq M_{g} \tilde{L}_{G}\norm{\varphi-\psi}_{H^{1}(-h,0;\mathbb{R})}
  \end{equation*}
  The second term can be dealt with as follows
  \begin{equation*}
    \begin{split}
      \abs{g(x_{2},\varphi(0)) - g(x_{2},\psi(0))}\abs{\e^{G(\psi)}} &\leq L_{g}\abs{\varphi(0) - \psi(0)}\e^{hM_{k}} \\
      &\leq L_{g}\norm{\varphi-\psi}_{\mathcal{C}([-h,0],\mathbb{R})} \e^{hM_{k}} \\
      &\leq L_{g} \e^{hM_{k}}(\abs{h}^{\frac{1}{2}}+\abs{h}^{-\frac{1}{2}}) \norm{\varphi-\psi}_{H^{1}(-h,0;\mathbb{R})}  \\
    \end{split}
  \end{equation*}
  where we used the Sobolev-embedding \Cref{Sobolev} again.
\end{proof}

We compartmentalize the main steps of our main theorem into two lemmata.
\begin{lemma}
  Let $\alpha>0$, $\delta>0$ and $\Phi\in H^{1}(-h,0;\mathbb{R}^{2})$. Then the first component $F_{1}$ of $F$ from \Cref{IVP} is Lipschitz-continuous on $\mathcal{B}_{\mathcal{C}([-h,0];\mathbb{R}^{2})}(\Phi,\delta)$.
\end{lemma}
\begin{proof}
  For $(\zeta,\varphi),(\eta,\psi)\in \mathcal{B}_{\mathcal{C}([-h,0];\mathbb{R}^{2})}(\Phi,\delta)$ we can estimate
  \begin{equation*}
      \abs[\big]{q(\varphi(0))\zeta(0)-q(\psi(0))\eta(0)}
      \leq \abs[\big]{q(\varphi(0))-q(\psi(0))} \abs[\big]{\zeta(0)} + \abs[\big]{q(\psi(0))}\abs[\big]{\zeta(0)-\eta(0)}
  \end{equation*}
  For any $\xi,\omega \in \mathcal{B}_{\mathcal{C}([-h,0],\mathbb{R})}(\Phi_{1},\delta)$ we can estimate using $M\coloneqq \norm{\Phi}_{\mathcal{C}([-h,0],\mathbb{R}^{2})}+\delta$, $M_{q}\coloneqq \text{sup}_{\sigma \in [\Phi_{1}(0)-M,\Phi_{1}(0)+M]} \abs{q(\sigma)}$ and $L_{q}\coloneqq \norm{q|_{[\Phi_{1}(0)-M,\Phi_{1}(0)+M]}}_{\text{Lip}}$ that: $\abs{\xi(0)}\leq M$ and $\abs{q(\xi(0))-q(\omega(0))}\leq L\abs{\xi(0)-\omega(0)}$. Hence we can further estimate:
  \begin{equation*}
    \begin{split}
      &\abs[\big]{q(\varphi(0))-q(\psi(0))}\abs[\big]{\zeta(0)} + \abs[\big]{q(\psi(0))}\abs{\zeta(0)-\eta(0)}\\
      &\leq L_{q}M\abs{\varphi(0)-\psi(0)} + M_{q}\abs{\zeta(0)-\eta(0)}\\
      &\leq L_{q}M\norm{\varphi-\psi}_{\mathcal{C}([-h,0],\mathbb{R})} + M_{q}\norm{\zeta-\eta}_{\mathcal{C}([-h,0],\mathbb{R})}\\
      &\leq (L_{q}M+M_{q})(\abs{h}^{\frac{1}{2}}-\abs{h}^{-\frac{1}{2}})\left(\norm{\varphi-\psi}_{H^{1}(-h,0;\mathbb{R})} + \norm{\zeta-\eta}_{H^{1}(-h,0;\mathbb{R})}\right)\\
      &=C\norm{(\zeta,\varphi)-(\eta,\psi)}_{H^{1}(-h,0;\mathbb{R}^{2})}
    \end{split}
  \end{equation*}
  where we used the Sobolev-embedding \Cref{Sobolev}.
\end{proof}

Similarly but sligthly more involved we also obtain:
\begin{lemma}
  Let $\alpha>0$, $\delta>0$ and $\Phi\in H^{1}(-h,0;\mathbb{R}^{2})$. Then the second component $F_{2}$ of $F$ from \Cref{IVP} is Lipschitz-continuous on $V_{\alpha}\cap \mathcal{B}_{\mathcal{C}([-h,0],\mathbb{R}^{2})}(\Phi,\delta)$.
\end{lemma}
\begin{proof}
  For $(\zeta,\varphi),(\eta,\psi)\in V_{\alpha}\cap \mathcal{B}_{\mathcal{C}([-h,0],\mathbb{R}^{2})}(\Phi,\delta)$ we can estimate
  \begin{equation*}
    \begin{split}
      &\abs[\big]{\beta(\varphi(-\tau (\varphi)))\zeta(-\tau(\varphi))\mathcal{G}(\varphi)-\mu\varphi(0) - \beta(\psi (-\tau (\psi)))\eta(-\tau(\psi))\mathcal{G}(\psi)+\mu\psi(0)}\\
      &\leq \abs[\big]{\beta(\varphi (-\tau (\varphi)))\zeta(-\tau(\varphi))\mathcal{G}(\varphi) - \beta(\varphi (-\tau (\varphi)))\zeta(-\tau(\varphi))\mathcal{G}(\psi)}\\
      &\quad+ \abs[\big]{\beta(\varphi (-\tau (\varphi)))\zeta(-\tau(\varphi))\mathcal{G}(\psi) - \beta(\psi (-\tau (\psi)))\eta(-\tau(\psi))\mathcal{G}(\psi)} + \mu\abs{\varphi(0)-\psi(0)}\\
      &\leq \abs[\big]{\beta(\varphi (-\tau (\varphi)))\zeta(-\tau(\varphi))\left[\mathcal{G}(\varphi)-\mathcal{G}(\psi)\right]} + \mu\abs{\varphi(0)-\psi(0)}\\
      &\quad+ \abs[\big]{\mathcal{G}(\psi)\beta(\varphi (-\tau (\varphi)))\left[\zeta(-\tau(\varphi)) - \eta(-\tau(\psi))\right]}\\
      &\quad+ \abs[\big]{\mathcal{G}(\psi)\eta(-\tau(\psi))\left[\beta(\varphi (-\tau (\varphi))) - \beta(\psi (-\tau (\psi)))\right]}
    \end{split}
  \end{equation*}
  We have to estimate four terms.
  \begin{enumerate}[leftmargin=3ex]
  \item First the easiest:
  \begin{equation*}
      \mu\abs{\varphi(0)-\psi(0)} \leq \mu \norm{\varphi-\psi}_{\mathcal{C}([-h,0],\mathbb{R})} \leq \mu (\abs{h}^{\frac{1}{2}}-\abs{h}^{-\frac{1}{2}}) \norm{\varphi-\psi}_{H^{1}(-h,0;\mathbb{R})}
  \end{equation*}

    \item For the second term notice that the image of $\mathcal{B}_{\mathcal{C}([-h,0],\mathbb{R})}(\phi,\delta)$ for $\phi\in\{\Phi_{1},\Phi_{2}\}$ under the map $\text{ev}\circ (\operatorname{id}\times(-\tau))$ is a bounded subset of $\mathbb{R}$: For $\zeta \in \mathcal{B}_{\mathcal{C}([-h,0],\mathbb{R})}(\phi,\delta)$ we can estimate
          \begin{equation*}
            \abs{\text{ev}\circ(\text{id}\times(-\tau))(\zeta)}=\abs{\text{ev}(\zeta,-\tau(\zeta))}=\abs{\zeta(-\tau(\zeta))}\leq \norm{\phi}_{\mathcal{C}([-h,0],\mathbb{R})}+\delta
          \end{equation*}
          Therefore the continuous function $\beta$ allows the supremum
          \begin{equation*}
            M_{\beta}\coloneqq \text{sup}\left\{\abs[\big]{\beta(\text{ev}\circ(\text{id}\times(-\tau)))(\xi)}\colon \, \xi\in \mathcal{B}_{\mathcal{C}([-h,0],\mathbb{R})}(\Phi_{2},\delta)\right\}<\infty.
          \end{equation*}
          With the term $M\coloneqq \norm{\Phi}_{\mathcal{C}([-h,0],\mathbb{R}^{2})}+\delta$ we can estimate:
  \begin{equation*}
    \begin{split}
      \abs[\big]{\beta(\varphi (-\tau (\varphi)))\zeta(-\tau(\varphi))\left[\mathcal{G}(\varphi)-\mathcal{G}(\psi)\right]}
      &=\abs[\big]{\beta(\varphi (-\tau (\varphi)))}\abs[\big]{\zeta(-\tau(\varphi))}\abs[\big]{\mathcal{G}(\varphi)-\mathcal{G}(\psi)}\\
      &\leq MM_{\beta} \abs{\mathcal{G}(\varphi)-\mathcal{G}(\psi)}\\
      &\leq MM_{\beta} L_{\mathcal{G}}\norm{\varphi-\psi}_{H^{1}(-h,0;\mathbb{R})}
    \end{split}
  \end{equation*}
  where we made use of the Lipschitz-continuity of $\mathcal{G}$ from \Cref{GLip}.
    \item We move on to the third term. We make use of the bound $M_{\beta}$ from the previous step again and further note that the function $\mathcal{G}$ admits a bound $M_{\mathcal{G}}$ on $\mathcal{B}_{\mathcal{C}([-h,0],\mathbb{R})}(\Phi_{2},\delta)$ since on that set it is Lipschitz-continuous by \Cref{GLip}. By Lipschitz-continuity of the delay $\tau$ from \Cref{tauLip} we also have a Lipschitz-constant $L_{\tau}$ of $\tau$. Then we can estimate:
    \begin{equation*}
    \begin{split}
      &\abs[\big]{\mathcal{G}(\psi)\beta(\varphi (-\tau (\varphi)))\left[\zeta(-\tau(\varphi)) - \eta(-\tau(\psi))\right]}\\
      &\leq M_{\mathcal{G}}M_{\beta}\abs[\big]{\zeta(-\tau(\varphi)) - \eta(-\tau(\psi))}\\
      &\leq M_{\mathcal{G}}M_{\beta}\abs[\big]{\zeta(-\tau(\varphi)) - \eta(-\tau(\varphi))} + M_{\mathcal{G}}M_{\beta} \abs[\big]{\eta(-\tau(\varphi)) - \eta(-\tau(\psi))} \\
      &\leq M_{\mathcal{G}}M_{\beta} \norm{\zeta - \eta}_{\mathcal{C}([-h,0],\mathbb{R})} + M_{\mathcal{G}}M_{\beta}\alpha \abs{\tau(\varphi)-\tau(\psi)}\\
      &\leq M_{\mathcal{G}}M_{\beta}(h^{\frac{1}{2}}+h^{-\frac{1}{2}}) \norm{\zeta - \eta}_{H^{1}(-h,0;\mathbb{R})} + M_{\mathcal{G}}M_{\beta} \alpha L_{\tau} \norm{\varphi-\psi}_{H^{1}(-h,0;\mathbb{R})}
    \end{split}
    \end{equation*}
  \item That only leaves one more term. As argued in step (2) the image of $\mathcal{B}_{\mathcal{C}([-h,0],\mathbb{R})}(\Phi_{2},\delta)$ under the map $\text{ev}\circ(\text{id}\times(-\tau))$ is a bounded subset of $\mathbb{R}$ and therefore the locally Lipschitz-continous $\beta$ admits a Lipschitz-constant $L_{\beta}$ on $\mathcal{B}_{\mathcal{C}([-h,0],\mathbb{R})}(\Phi_{2},\delta)$. We calculate
  \begin{equation*}
    \begin{split}
      &\abs[\big]{\mathcal{G}(\psi)\eta(-\tau(\psi))\left[\beta(\varphi (-\tau (\varphi))) - \beta(\psi (-\tau (\psi)))\right]} \\
      &\leq MM_{\mathcal{G}} \abs[\big]{\beta(\varphi (-\tau (\varphi))) - \beta(\psi (-\tau (\psi)))} \\
      &\leq MM_{\mathcal{G}} L_{\beta} \abs[\big]{\varphi(-\tau(\varphi))-\psi(-\tau(\psi))} \\
      &= MM_{\mathcal{G}} L_{\beta} \abs[\big]{\text{ev}\circ(\text{id}\times(-\tau))(\varphi -\psi)} \\
      &\leq MM_{\mathcal{G}} L_{\beta} \norm{\varphi -\psi}_{H^{1}(-h,0;\mathbb{R})}
    \end{split}
  \end{equation*}
  where we used \Cref{evtauLip} in the last step.
  \end{enumerate}
  That accounts for all terms.
\end{proof}

Now we move on to the proof of the main theorem:
\begin{proof}[Proof of \Cref{localExistence}]
  We need only apply \Cref{localFDE} in conjunction with \Cref{Restriction}; in other words we have to prove that $F$ is almost uniformly Lipschitz-continuous on a suitable $\norm{.}_{\infty}$-neighbourhood of $\Phi$. So for a given $\alpha>0$ and a given $\delta>0$ we restrict ourselves to $A\coloneqq \mathcal{B}_{\mathcal{C}([-h,0],\mathbb{R}^{2})}(\Phi,\delta)\cap V_{\alpha}$. Let $(\zeta,\varphi),(\eta,\psi)\in A$. Then we can estimate:
  \begin{equation*}
    \begin{split}
      &\norm{F(\zeta,\varphi) - F(\eta,\psi)}_{H^{1}(-h,0;\mathbb{R}^{2})}^{2} \\
      &= \norm{F_{1}(\zeta,\varphi) - F_{1}(\eta,\psi)}_{H^{1}(-h,0;\mathbb{R})}^{2} + \norm{F_{2}(\zeta,\varphi) - F_{2}(\eta,\psi)}_{H^{1}(-h,0;\mathbb{R})}^{2}\\
      &\leq L_{1}^{2}\left(\norm{\zeta-\eta}_{H^{1}(-h,0;\mathbb{R})}^{2}+\norm{\varphi-\psi}_{H^{1}(-h,0;\mathbb{R})}^{2}\right) \\
      &\quad+ L_{2}^{2}\left(\norm{\zeta-\eta}_{H^{1}(-h,0;\mathbb{R})}^{2}+\norm{\varphi-\psi}_{H^{1}(-h,0;\mathbb{R})}^{2}\right)\\
      &= (L_{1}^{2}+L_{2}^{2})\norm{(\zeta,\varphi)-(\eta,\psi)}_{H^{1}(-h,0;\mathbb{R}^{2})}^{2}
    \end{split}
  \end{equation*}
  where we denote the Lipschitz constants of the first and second components by $L_{1}$ and $L_{2}$ from the preceding two lemmata respectively. This proves the claim.
\end{proof}

With our main result we can also show global existence of solutions thanks to an exponential a priori bound on the solution. To prove this we make the following observation first:
\begin{remark}[Variations of constants formula]
  \label{VoC}
  For some $\Phi=(\varphi,\psi)\in H^{1}(-h,0;\mathbb{R}^{2})$ the solution $(w,v)$ of \Cref{IVP} by the variation of constants formula satisfies
  \begin{equation*}
    \begin{split}
      w(t)&=\varphi(0)\e^{\int_{0}^{t}q(v(s))\dx[s]} \\
      v(t)&=\e^{-\mu t}\left[\psi(0)+\int_{0}^{t}\e^{\mu s}\beta(v(s-\tau(v_{s})))\mathcal{G}(v_{s})w(s-\tau(v_{s}))\dx[s]\right]
    \end{split}
  \end{equation*}
  The formula is easily verified via calculation, which we omit.
\end{remark}
\begin{theorem}[Global existence of solutions]
  \label{globalExistence}
  Under the additional \Cref{Furhter assumptions} the uniquely defined local solution of \Cref{IVP} is global.\footnote{Meaning the maximal existence time is $T=\infty$.} If additionally $\beta$ is a bounded function, the global solution of \Cref{IVP} is exponentially bounded.
\end{theorem}
\begin{proof}
  We can solve \Cref{IVP} locally because of \Cref{localExistence}. Proceeding iteratively we start on $V_{\alpha}\cap \mathcal{B}_{\mathcal{C}([-h,0],\mathbb{R}^{2})}(\Phi,\delta)$ for some $\alpha>0$ and $\delta>0$. We then apply \Cref{localExistence} and obtain a solution $u=(w,v)$ up to some time $T_{0}\leq T$. Then we can take a new prehistory (up to $T_{0}$ instead of $0$) and a larger $\alpha_{\text{new}}\coloneq 2\alpha$ as well as some $\delta_{\text{new}}\geq \delta$ large enough to accomodate for the new prehistory and extend further. Iterating this process we arrive at maximal existence time $T$ or if we do not, then $\norm{u}_{\infty}$ or $\norm{u}_{\text{Lip}}$ approach infinity before reaching $T$. We show that neither can happen in finite time. Also note in this context that $y$ is well-defined as long as $\norm{u}_{\infty}$ remains finite by \Cref{yexpbdd}.
  \begin{itemize}[leftmargin=3ex]
    \item To exclude the first outcome we show that $u(t)$ exists for all large times $t$ and $u$ is even exponentially bounded if $\beta$ is bounded. For this consideration let $u$ be a maximal solution of the DDE \Cref{IVP}. We use the variation of constants formula from \Cref{VoC} and obtain from the first equation using the global bound $q\leq M_{q}$:
    \begin{equation*}
      \abs{w(t)}=\abs{\varphi(0)}\e^{\int_{0}^{t}q(v(s))\dx[s]}\leq \abs{\varphi(0)}\e^{\int_0^t M_{q}\dx[s]} = \abs{\varphi(0)}\e^{t M_{q}}
    \end{equation*}
          Before we can estimate $v(t)$ we want to estimate the function $\mathcal{G}$. For this we use the global bounds $g\leq K$ and $d-D_{1}g\leq M_{k}$:
          \begin{align*}
            \abs{\mathcal{G}(v_{t})}&=\abs{g(x_{2},v(t))}\e^{\int_0^{\tau(v_{t})} (d-D_{1}g)(y(s,v_{t}),v(t-s))\dx[s]}\\
                                    &\leq K\e^{\int_0^{\tau(v_{t})} M_{k} \dx[s]}\\
                                    &\leq K\e^{h M_{k}}\eqcolon M_{\mathcal{G}}
          \end{align*}
          Using the second equation from \Cref{VoC} we further calculate using the linear growth bound $\abs{\beta(x)}\leq C_{\beta}\abs{x}+a_{\beta}$:
    \begin{align*}
      \abs{v(t)}&=\e^{-\mu t}\abs[\bigg]{\psi(0)+\int_{0}^{t}\e^{\mu s}\beta(v(s-\tau(v_{s})))\mathcal{G}(v_{s})w(s-\tau(v_{s}))\dx[s]}\\
                &\leq \e^{-\mu t}\left[\abs{\psi(0)}+\int_{0}^{t}\e^{\mu s} \abs[\big]{\beta(v(s-\tau(v_{s})))\mathcal{G}(v_{s})w(s-\tau(v_{s}))}\dx[s]\right]\\
                &\leq \e^{-\mu t}\left[\abs{\psi(0)}+\int_{0}^{t}\e^{\mu s} \left(C_{\beta}\abs[\big]{v(s-\tau(v_{s}))}+a_{\beta}\right)M_{\mathcal{G}}\abs{\varphi(0)}\e^{(s-\tau(v_{s})) M_{q}}\dx[s]\right]\\
                &\leq \e^{-\mu t}\left[\abs{\psi(0)}+\abs{\varphi(0)}M_{\mathcal{G}} \int_{0}^{t}\e^{(\mu + M_{q}) s} \left(C_{\beta}\abs[\big]{v(s-\tau(v_{s}))}+a_{\beta}\right)\dx[s]\right]\\
                &\leq \e^{-\mu t}\left[\abs{\psi(0)}+\abs{\varphi(0)}M_{\mathcal{G}} \int_{0}^{t}\e^{(\mu + M_{q}) s} \left(C_{\beta}\underset{\kappa \in [0,s]}{\operatorname{sup}}\abs{v(\kappa)}+ \norm{\psi}_{\mathcal{C}([-h,0],\mathbb{R})} +a_{\beta}\right)\dx[s]\right]\\
                &\leq \e^{-\mu t}\abs{\psi(0)}+\abs{\varphi(0)}M_{\mathcal{G}}\e^{-\mu t}\left[\tfrac{a_{\beta}+\norm{\psi}_{\infty}}{\mu + M_{q}}\e^{(\mu + M_{q})t}+ \int_{0}^{t}\e^{(\mu + M_{q}) s} C_{\beta}\underset{\kappa \in [0,s]}{\operatorname{sup}}\abs{v(\kappa)}\dx[s]\right]\\
                &\leq C\e^{M_{q} t} + \abs{\varphi(0)}M_{\mathcal{G}}C_{\beta} \int_{0}^{t}\e^{(\mu + M_{q}) s} \underset{\kappa \in [0,s]}{\operatorname{sup}}\abs{v(\kappa)}\dx[s]
    \end{align*}
          Using Gr\"{o}nwall's lemma we obtain:
          \begin{equation*}
            \abs{v(t)}\leq \underset{\kappa \in [0,t]}{\operatorname{sup}}\abs{v(\kappa)}\leq C \e^{M_{q} t} + \int_0^t C \e^{M_{q} s}\abs{\varphi(0)} \e^{(\mu + M_{q}) s} C_{\beta}M_{\mathcal{G}}\e^{\int_s^t \abs{\varphi(0)}C_{\beta}M_{\mathcal{G}} \e^{(\mu + M_{q}) \sigma} \dx[\sigma]} \dx[s]
          \end{equation*}
          This shows that $v(t)$ exists for all times. Under the global bound $\abs{\beta}\leq M_{\beta}$ for $\beta$ we obtain the simpler estimate:
          \begin{align*}
      \abs{v(t)}&=\e^{-\mu t}\abs[\bigg]{\psi(0)+\int_{0}^{t}\e^{\mu s}\beta(v(s-\tau(v_{s})))\mathcal{G}(v_{s})w(s-\tau(v_{s}))\dx[s]}\\
                &\leq \e^{-\mu t}\left[\abs{\psi(0)}+\int_{0}^{t}\e^{\mu s}\abs[\big]{\beta(v(s-\tau(v_{s})))\mathcal{G}(v_{s})w(s-\tau(v_{s}))}\dx[s]\right]\\
                &\leq \e^{-\mu t}\abs{\psi(0)}+\int_{0}^{t}\e^{\mu s}M_{\beta}M_{\mathcal{G}}\abs{\varphi(0)}\e^{(s-\tau(v_{s})) M_{q}}\dx[s]\\
            &\leq c + C\e^{(\mu + M_{q})t}
          \end{align*}
  \item For the derivative we can estimate:
  \begin{alignat*}{2}
    \abs{u'(t)}^{2}&=\abs{F(u_{t})}^{2} \\
                   &= \abs[\big]{q(v_{t}(0))w_{t}(0)}^{2} + \abs[\big]{\beta(v_{t}(-\tau(v_{t})))w_{t}(-\tau(v_{t}))\mathcal{G}(v_{t}) - \mu v_{t}(0)}^{2}\\
                   &=\abs[\big]{q(v(t))w(t)}^{2} + \abs[\big]{\beta(v(t-\tau(v_{t})))w(t-\tau(v_{t}))\mathcal{G}(v_{t}) - \mu v(t)}^{2}
  \end{alignat*}
          Using the terminology from the previous part we can estimate:
          \begin{alignat*}{2}
            \abs{u'(t)}^{2}&=\abs[\big]{q(v(t))w(t)}^{2} + \abs[\big]{\beta(v(t-\tau(v_{t})))w(t-\tau(v_{t}))\mathcal{G}(v_{t}) - \mu v(t)}^{2}\\
                           &\leq M_{q}^{2}\abs[\big]{\varphi(0)}^{2}\e^{t 2 M_{q}} + 2\abs[\big]{\beta(v(t-\tau(v_{t})))w(t-\tau(v_{t}))\mathcal{G}(v_{t})}^{2} + 2 \mu^{2} \abs{v(t)}^{2}\\
            &\leq M_{q}^{2}\abs{\varphi(0)}^{2}\e^{t 2 M_{q}} + 2\abs[\big]{\beta(v(t-\tau(v_{t})))}^{2}M_{q}^{2}\abs{\varphi(0)}^{2}\e^{t 2 M_{q}}M_{\mathcal{G}}^{2} + 2 \mu^{2} \abs{v(t)}^{2}
          \end{alignat*}
          This expression exists for all large $t$.
  \end{itemize}
\end{proof}
Hence we have established local and global existence as well as uniqueness for \Cref{IVP}. To complete the discussion of well-posedness of \Cref{IVP} we refer to the next section.

\section{Outlook and Discussion}
The presented results improve on previous theory. There are two major articles of comparison. We have already pointed out the differences to \cite{Getto2016}. Here our results appear first as a weakening of regularity in the assumptions from continuous differentiable to (locally) Lipschitz-continuous functions. In practice, not having to verify continuous differentiability and instead being content with the Lipschitz-conditions stated in our \Cref{Assumptions} is a considerable simplification though. Second and more importantly we do not require compatibility of the initial prehistory with the right-hand side of \Cref{IVP} as one usually does if one pursues the solution manifold approach (as in \cite{Getto2016}). Since we are able to prove existence of infinite-time solutions this article implies in particular that the solution manifold (as introduced in \cite{Walther2003}) is not empty and the solution $u$ we obtain is still a continuously differentiable function as in the classical theory. Thus we benefit from all the results built on the concept of the solution manifold as well.\\
The other major article investigating the same stame cell model is \cite{Getto2021}. Here the results are remarkably similar to ours in terms of assumptions:
\begin{assumptions}[Assumptions of {\cite[thm.~4.10]{Getto2021}}]
  \label{AssumptionsGetto}
  Let $b,\epsilon,K\in \mathbb{R}$ satisying $x_{2}-x_{1}\in \left(0,\epsilon\frac{b}{K}\right)$, let $\mathcal{D}_{g}\coloneq\overline{\mathcal{B}(x_{2},b)}\times \mathbb{R}_{+}$ and let
  \begin{enumerate}[leftmargin=3ex]
    \item $q\colon \, \mathbb{R}_{+}\rightarrow \mathbb{R}$ be locally Lipschitz-continuous.
    \item $\beta\colon \, \mathbb{R}_{+}\rightarrow \mathbb{R}_{+}$ be locally Lipschitz-continuous.
    \item $d\colon \mathcal{D}_{g}\rightarrow \mathbb{R}$ be locally Lipschitz-continuous.
    \item $g\colon \,\mathcal{D}_{g}\rightarrow [\epsilon,K]$ satisfy the conditions
  \begin{itemize}[leftmargin=3ex]
    \item $g$ is locally Lipschitz-continuous and partially differentiable with respect to the first component.
    \item $D_{1}g$ is locally Lipschitz-continuous.
  \end{itemize}
  \end{enumerate}
  Let $h\coloneqq\frac{b}{K}$.
\end{assumptions}
The difference in  assumptions in terms of regularity is miniscule: Essentially we do require $g$ to be globally Lipschitz-continuous but $D_{1}g$ only locally Lipschitz-continuous with respect to one component, whereas in \cite{Getto2021} local Lipschitz-continuity suffices for $g$, but $D_{1}g$ needs to be locally Lipschitz with respect to both arguments. There is a more important difference here though: That of positivity. The corresponding local existence theorem reads:
\begin{theorem*}[{\cite[thm.~4.10]{Getto2021}}]
  Under the \Cref{AssumptionsGetto} the IVP \Cref{IVP} has a unique local non-continuable solution through any given prehistory $\Phi\in \{\Psi\in \mathcal{C}\left([-h,0],\mathbb{R}^{2}_{+}\right)\colon \norm{\Psi}_{\operatorname{Lip}}<\infty\}$.
\end{theorem*}
The regularity assumption on $\Phi$ is very close to our own, the positivity assumption though is superfluous in our approach. The article \cite{Getto2021} relies heavily on contraction arguments appealing to a result of relative compactness of orbits of certain dynamical systems from \cite{Hale1969}, which we do not require at all. In fact our approach follows the well-trod footsteps of classical works in the field of ODEs, making it very accessible.\\
The major benefit of our approach is the use of weak solution theory that shortens proofs considerably and that one does not have to verify classical differentiability (of in general Banach space valued functions) or apply the implicit function theorem (for Banach spaces). Instead one simply applies a generalized version of the Picard-Lindel\"{o}f theorem. The mathematics required for the proof of that theorem rests on the theory of evolutionary equations in which differential equations are solved in exponentially weighted $L^{p}$-spaces giving rise to a time derivative with shifted spectrum that allows inversion. For more pertinent results in the context of DDEs we refer to \cite{Waurick2023}, for the general theory to \cite{Waurick2022}. The authors plan to apply this approach to PDEs with delay in the future.\\
In this article we have argued existence and uniqueness of solutions for the IVP \Cref{IVP}. For the classical property of well-posedness of a problem we have left out the discussion of continuous dependence of prehistories. In fact since the solution theory in \cite{Waurick2023} is based on the contraction mapping principle a corresponding continuous dependence result can be derived in a standard way, cf. \cite{Kalauch2014}.

The purpose of this article was thus twofold: For one, the reduction of assumptions of a concrete example with practical implications. On the other hand this model equation from cell biology serves as a nice demonstration of the theory we have developped so far applied to a non-trivial and non-standard example that allows us to present the benefits.\\

\bibliographystyle{abbrvurl}
\bibliography{references}

\end{document}